\documentclass[11pt, reqno]{amsart}
\usepackage{latexsym,amsmath,amssymb, amsthm, mathtools}
\usepackage{cases, verbatim}
\usepackage[colorlinks]{hyperref}
\usepackage{xcolor}
\usepackage{float}

\usepackage{mathscinet}

\hypersetup{
    unicode=false,          
    pdftoolbar=true,        
    pdfmenubar=true,        
    pdffitwindow=false,     
    colorlinks=false,       
    linkcolor=red,          
    linkbordercolor=red,
    citecolor=green,        
    citebordercolor=green,
    filecolor=magenta,      
    urlcolor=cyan,           
   urlbordercolor={1 1 1}  
}

\usepackage[margin=1.5in]{geometry}

\newtheorem{thm}{Theorem}
\newtheorem{lem}[thm]{Lemma}
\newtheorem{prop}[thm]{Proposition}
\newtheorem{cor}[thm]{Corollary}

\theoremstyle{remark}
\newtheorem{rmk}[thm]{Remark}
\newtheorem{example}[thm]{Example}

\theoremstyle{definition}
\newtheorem{defi}[thm]{Definition}

\numberwithin{thm}{section} 
\numberwithin{equation}{section}

\makeatletter

\newcommand{\Rmnum}[1]{\expandafter\@slowromancap\romannumeral #1@}
\makeatother

\def\R{{\mathbb R}}

\def\N{{\mathcal N}}

\newcommand{\Oba}{\overline{\Omega}}

\newcommand{\vep}{\varepsilon}
\newcommand{\ol}{\overline}

\newcommand{\bpm}{\begin{pmatrix}}
\newcommand{\epm}{\end{pmatrix}}
\newcommand{\la}{\left\langle}
\newcommand{\ra}{\right\rangle}

\newcommand{\beq}{\begin{equation}}
\newcommand{\eeq}{\end{equation}}

\def\dist{{\rm dist\,}}

plus 6pt minus 12pt
\title[Convexity for Neumann problems]{Convexity preservation for first-order nonlinear evolution equations with Neumann-type boundary conditions}

\author[D.~Lin]{Daowen Lin}
\address[Daowen Lin]{Geometric Partial Differential Equations Unit, Okinawa Institute of Science and Technology Graduate University, Okinawa 904-0495, Japan}

\email{daowen.lin@oist.jp}

\author[Q.~Liu]{Qing Liu}
\address[Qing Liu]{Geometric Partial Differential Equations Unit, Okinawa Institute of Science and Technology Graduate University, Okinawa 904-0495, Japan}
\email{qing.liu@oist.jp}

\date{}

\begin{document}

\begin{abstract}
We study preservation of spatial convexity for viscosity solutions of first order Hamilton-Jacobi evolution equations with Neumann-type boundary conditions in
bounded convex domains. Under a monotonicity assumption on the Hamiltonian in
the outward normal direction, every viscosity supersolution of the Neumann
problem is also a state-constraint supersolution. This reduces the boundary
difficulty to a state-constraint problem and enables us to prove convexity
preservation via the convex-envelope argument established by Alvarez, Lasry and Lions (1997).
We also obtain propagation estimates for semiconvexity and strong convexity
under similar structural assumptions.
\end{abstract}

\subjclass[2020]{35F21, 35B35, 35B65, 49L25}
\keywords{Hamilton-Jacobi equations, convexity preservation, Neumann boundary conditions, state constraints, viscosity solutions}

\maketitle

\section{Introduction}

In this paper, we study the convexity preserving property of viscosity solutions to a class of nonlinear evolution equations with Neumann-type boundary conditions. Suppose that $\Omega\subset \R^n$ is a bounded convex $C^2$ domain and $\nu$ denotes the outward unit normal to $\partial\Omega$. The general equation we shall consider is of the form
\begin{numcases}{}
    u_t+H(x,t,u,\nabla u)=0
    &\quad \text{in $\Omega\times(0,\infty)$,}\label{general eq1}\\
    \la \nabla u,\nu\ra=g(x)
    &\quad \text{on $\partial\Omega\times(0,\infty)$,}\label{general eq2}\\
    u(\cdot,0)=u_0
    &\quad \text{in $\Oba$,}\label{initial}
\end{numcases}
where $H:\Oba\times[0,\infty)\times\R\times\R^n\to\R$ is continuous, $g\in C^2(\partial\Omega)$, and $u_0\in C(\Oba)$ is a given initial value. Further assumptions on $H$ will be introduced below.

Suppose that there exists a unique viscosity solution $u\in C(\Oba\times[0,\infty))$
of \eqref{general eq1}--\eqref{initial}. Our main purpose is to give structural conditions under which convexity of the initial data is preserved by the evolution: if $u_0$ is convex in $\Oba$, then $u(\cdot,t)$ remains convex in $\Oba$ for all $t\geq0$. Convexity of classical and viscosity solutions has been extensively studied for elliptic and parabolic equations posed either in the whole space or with Dirichlet boundary conditions; see, for example, \cite{Kor,KaBook,Ke,GGIS,Ju} and the references therein.
Still within the settings of Cauchy problems in the whole space or Dirichlet problems, more recently the preservation and loss of power concavity for the heat equation and related parabolic equations were investigated by Chau and Weinkove \cite{CW1,CW2} and by Ishige, Salani, and Takatsu \cite{IST1, IST2, IST3}. Concavity preservation for diffusion equations is closely linked to geometric and functional inequalities; see \cite{BrLi,Borell1,Borell2,BCE,ILSBBL}. Convexity preservation for the Cauchy problem for first order Hamilton-Jacobi equations has also been studied in the more general setting of geodesic metric spaces; see \cite{LN}.

The aforementioned works motivate us to explore how boundary conditions affect such convexity or concavity properties. For second order state-constraint problems, Alvarez, Lasry, and Lions \cite{ALL} established convexity of viscosity solutions. Andrews, Clutterbuck, and Hauer \cite{ACH} and Crasta and Fragal\`a \cite{CrFr6} studied log-concavity for Robin problems associated with the Laplacian. In \cite{HL2}, convexity of level sets is established for level-set curvature flow with a dynamic boundary condition, using the game-theoretic approach of \cite{LSZ}. To the best of our knowledge, convexity preservation for nonlinear evolution equations with Neumann-type boundary conditions has not been systematically investigated. In the presence of second order terms, the interaction between diffusion and the boundary makes convexity preservation considerably more delicate. We therefore focus mainly on first order equations, while a one-dimensional second order problem is discussed separately in Section~\ref{sec:parabolic}.

Typical examples of \eqref{general eq1} that we have in mind are
\begin{equation}\label{ex1-eq}
u_t+|\nabla u|=0
\quad\text{in $\Omega\times(0,\infty)$,}
\end{equation}
and
\begin{equation}\label{ex2-eq}
u_t+\left(1+|\nabla u|^2\right)^{1/2}=0
\quad\text{in $\Omega\times(0,\infty)$.}
\end{equation}
Equation \eqref{ex1-eq} is of eikonal type, and the associated homogeneous Neumann problem 
arises naturally in deterministic optimal control for systems with reflection at the boundary. It also provides a level set formulation for the evolution of a hypersurface moving with unit normal speed and meeting $\partial\Omega$ at a right angle. Equation \eqref{ex2-eq}, together with the homogeneous Neumann boundary condition, describes the corresponding graph-type evolution. 

For first order fully nonlinear problems such as \eqref{ex1-eq} and \eqref{ex2-eq}, the boundary condition need not hold in the classical sense and is naturally interpreted in the viscosity sense. At a boundary point $(x,t)\in\partial\Omega\times(0,\infty)$, if a viscosity subsolution is touched from above by $\varphi\in C^1(\Oba\times(0,\infty))$, then one requires
\begin{equation}\label{intro-vis-boundary-sub}
\min\left\{
\varphi_t+H(x,t,u,\nabla\varphi),\
\la\nabla\varphi,\nu(x)\ra-g(x)
\right\}\leq0.
\end{equation}
For a viscosity supersolution touched from below, one requires
\begin{equation}\label{intro-vis-boundary-super}
\max\left\{
\varphi_t+H(x,t,u,\nabla\varphi),\
\la\nabla\varphi,\nu(x)\ra-g(x)
\right\}\geq0.
\end{equation}
Thus, at the boundary, either the interior equation or the Neumann condition is imposed in the viscosity sense. We refer to \cite{Lio1,I91,BL,CIL} for the viscosity theory of Hamilton-Jacobi equations with Neumann or more general oblique boundary conditions, and to \cite{IshiiWKAM,BIM} for weak KAM theory and large-time behavior in this setting. Throughout the paper, we work in a framework in which existence and comparison for \eqref{general eq1}--\eqref{initial} are available. The precise definition of viscosity solutions and a convenient set of hypotheses guaranteeing comparison will be briefly reviewed in Section~\ref{sec:prelim}.

Within this framework of viscosity solutions, we are able to investigate the convexity preserving property for \eqref{general eq1}--\eqref{initial}. Our study is inspired by the following elementary example.  

\begin{example}\label{exa1}
Let $\Omega=B_1(0)\subset\R^n$, let $g=0$, and consider \eqref{ex1-eq} with
\begin{equation}\label{ex1-sol}
u_0(x)=|x|,\qquad x\in\ol{B_1(0)}.
\end{equation}
Then
\[
u(x,t)=\max\{|x|-t,0\},
\qquad x\in\ol{B_1(0)},\quad t\geq0,
\]
is the unique viscosity solution. For $0<t<1$, the homogeneous Neumann condition is satisfied in the viscosity sense rather than in the classical sense. The function $u(\cdot,t)$ is convex for every $t\geq0$. Similar convexity-preserving examples can be constructed for \eqref{ex2-eq}.

In contrast, if we replace \eqref{ex1-eq} by
\begin{equation}\label{ex1-eq2}
u_t-|\nabla u|=0
\quad\text{in $\Omega\times(0,\infty)$}
\end{equation}
with the same domain, initial data, and homogeneous Neumann condition, then the unique viscosity solution is
\begin{equation}\label{ex1-sol2}
u(x,t)=\min\{|x|+t,1\},
\qquad x\in\ol{B_1(0)},\quad t\geq0,
\end{equation}
which is not convex in $x$ for any $0<t<1$.
\end{example}

The key difference between \eqref{ex1-eq} and \eqref{ex1-eq2} is their behavior at the boundary. For the first equation, the solution in Example~\ref{exa1} satisfies the supersolution inequality for the interior equation even at $\partial\Omega$, whereas this property fails for \eqref{ex1-eq2}. This suggests that, in order to preserve convexity, it is useful to require that the Neumann solution also satisfy the state-constraint supersolution condition
\begin{equation}\label{sc-form}
u_t+H(x,t,u,\nabla u)\geq0
\quad\text{on $\Oba\times(0,\infty)$}
\end{equation}
in the viscosity sense. In this case, we also say that $u$ is a constrained viscosity supersolution. Once \eqref{sc-form} is available, the convex-envelope argument for state-constraint problems developed by Alvarez, Lasry and Lions \cite{ALL} can be adapted to the present time-dependent setting.

A main observation of this paper is that the state-constraint property follows directly from the Neumann problem under a natural monotonicity condition on the Hamiltonian in the outward normal direction. More precisely, we impose the following assumptions:
\begin{enumerate}
\item[(H1)] There exists $\gamma\in\R$ such that
\[
H(x,t,r_1,p)-H(x,t,r_2,p)\geq\gamma(r_1-r_2)
\]
for all $x\in\Oba$, $t\geq0$, $p\in\R^n$ and $r_1,r_2\in\R$ satisfying $r_1\geq r_2$.

\item[(H2)] There exists a modulus of continuity $\omega_x$ such that
\[
|H(x_1,t,r,p)-H(x_2,t,r,p)|
\leq\omega_x\bigl(|x_1-x_2|(1+|p|)\bigr)
\]
for all $x_1,x_2\in\Oba$, $t\geq0$, $r\in\R$ and $p\in\R^n$.

\item[(H3)] There exists a modulus of continuity $\omega_p$ such that
\[
|H(x,t,r,p_1)-H(x,t,r,p_2)|
\leq\omega_p(|p_1-p_2|)
\]
for all $x\in\Oba$, $t\geq0$, $r\in\R$ and $p_1,p_2\in\R^n$.

\item[(H4)] For each $(t,p)\in[0,\infty)\times\R^n$, the map $(x,r)\longmapsto H(x,t,r,p)$ is concave. More precisely,
\begin{equation}\label{H-concave0}
\sum_{i=1}^{n+1}\lambda_i H(x_i,t,r_i,p)
\leq H(x,t,r,p)
\end{equation}
whenever
\begin{equation}\label{H-concave2}
x=\sum_{i=1}^{n+1}\lambda_i x_i,
\qquad
r=\sum_{i=1}^{n+1}\lambda_i r_i,
\end{equation}
where $x_i\in\Oba$, $r_i\in\R$, $\lambda_i\geq0$ and
$\sum_{i=1}^{n+1}\lambda_i=1$.

\item[(H5)] For every $(x,t,r)\in\partial\Omega\times(0,\infty)\times\R$, every $p\in\R^n$ satisfying $\la p,\nu(x)\ra\geq g(x)$,
and every $\alpha\geq0$, we have 
\begin{equation}\label{H-mono1}
H(x,t,r,p+\alpha\nu(x))
\geq H(x,t,r,p).
\end{equation}
\end{enumerate}

The assumptions (H1)--(H3) are standard structural conditions used in comparison arguments for viscosity solutions. Condition (H4) is  standard in studying convexity of solutions to fully nonlinear equations, used also in the convex-envelope method \cite{ALL}; see related results in \cite{GGIS, Ju, ILS}. The new boundary condition (H5) has a simple geometric meaning: once the normal component of the gradient reaches the Neumann threshold $g(x)$, the Hamiltonian does not decrease when the gradient is moved further in the outward normal direction. For homogeneous Neumann data $g=0$, condition (H5) is satisfied, for example, by Hamiltonians of the form
\[
H(x,t,r,p)=\widetilde H(x,t,r,|p|)
\]
for which $q\mapsto\widetilde H(x,t,r,q)$ is nondecreasing. Indeed, if $\la p,\nu\ra\geq0$, then
\[
|p+\alpha\nu|\geq|p|
\qquad\text{for every }\alpha\geq0.
\]
The same observation applies when $g\geq0$.

Our first main result is the following.

\begin{thm}[Convexity preserving property]\label{thm:convex0}
Let $\Omega\subset\R^n$ be a bounded convex domain of $C^2$ class with outward unit normal $\nu$. Suppose that
$H\in C(\Oba\times[0,\infty)\times\R\times\R^n)$ satisfies
(H1)--(H5). Let $u_0\in C(\Oba)$ be convex and let
$u\in C(\Oba\times[0,\infty))$ be a viscosity solution of
\eqref{general eq1}--\eqref{initial}. Then $u(\cdot,t)$ is convex in
$\Oba$ for every $t\geq0$.
\end{thm}

The key step is the Neumann-to-state-constraint principle. In Lemma~\ref{lem:constraint}, we show that, under (H5), every viscosity supersolution of the Neumann problem automatically satisfies \eqref{sc-form}. Once this is established, we can adapt the convex-envelope argument of \cite{ALL}. The original result in \cite{ALL} concerns the stationary problem; here we apply the same mechanism to first order evolution equations. Example~\ref{exa1} also shows why (H5) is a natural condition: it is satisfied by the Hamiltonian $H(p)=|p|$, but not by $H(p)=-|p|$. To the best of our knowledge, spatial convexity preservation for first order evolutionary Hamilton-Jacobi equations with Neumann boundary conditions has not been established previously. Although the argument becomes relatively direct once the Neumann-to-state-constraint conversion is available, we believe that this observation is worth recording, both in its own right and as a possible starting point for further developments.

Note that {\rm (H5)} provides a convenient sufficient condition for convexity preservation that connects to the state constraint problem and is independent of the initial data. In general, depend on the
particular initial value, it need not be necessary; see Example~\ref{exa:no-state-constraint}. It would
be an interesting future project to develop a more general theory in which {\rm (H5)} is weakened and structural conditions on the
initial data are incorporated.

We also study quantitative versions of convexity preservation under similar assumptions. For $C\in\R$, we say that $f\in C(\Oba)$ is $C$-convex if
\[
x\longmapsto f(x)-\frac C2|x|^2
\]
is convex in $\Oba$. Thus $C=0$ corresponds to ordinary convexity, $C>0$ to strong convexity, and $C<0$ to semiconvexity. For Hamiltonians independent of $u$, we use the following quantitative version of (H4):
\begin{enumerate}
\item[(H4$'$)] There exist $T>0$ and a continuous function $m:\R\to\R$ such that
\begin{equation}\label{H-concave1}
\sum_{i=1}^{n+1}\lambda_i
H(x_i,t,p+s(x_i-x))
\leq H(x,t,p)+m(s)Q
\end{equation}
for all $s\in\R$, $x_1,\ldots,x_{n+1}\in\Oba$, $t\in[0,T)$ and
$p\in\R^n$, where
\[
x=\sum_{i=1}^{n+1}\lambda_i x_i,
\qquad
Q=\sum_{i=1}^{n+1}\lambda_i|x_i|^2-|x|^2,
\qquad
\lambda_i\geq0,\quad \sum_{i=1}^{n+1}\lambda_i=1.
\]
\end{enumerate}

\begin{thm}[Propagation of generalized convexity]\label{thm:convex}
Let $\Omega\subset\R^n$ be a bounded convex domain of $C^2$ class with outward unit normal $\nu$. Suppose that the Hamiltonian is independent of $u$ and that
$H\in C(\Oba\times[0,\infty)\times\R^n)$ satisfies
(H2), (H3), (H4$'$), and (H5). Let $\mu\in C^1([0,T))$ satisfy
\begin{equation}\label{g-fun}
\frac12\mu'(t)+m(\mu(t))\leq0,
\qquad t\in(0,T).
\end{equation}
If $u_0\in C(\Oba)$ is $\mu(0)$-convex and
$u\in C(\Oba\times[0,T))$ is a viscosity solution of
\eqref{general eq1}--\eqref{initial}, then $u(\cdot,t)$ is
$\mu(t)$-convex in $\Oba$ for every $t\in[0,T)$.
\end{thm}

Theorem~\ref{thm:convex} simultaneously gives propagation of semiconvexity and strong convexity. If $\mu(0)<0$ and \eqref{g-fun} admits a solution that remains finite from below, then $\mu(t)$ gives a quantitative semiconvexity bound. If $\mu(0)>0$ and one can choose $\mu(t)>0$, then strong convexity is propagated. More general Hamiltonians depending on $u$ can also be treated under a suitable variant of (H4$'$); see Section~\ref{sec:convexity}.

Semiconvexity and semiconcavity estimates for Hamilton-Jacobi equations in the Cauchy problem have been studied extensively. As discussed in \cite{Lions82,CS}, two classical approaches are vanishing viscosity, combined with Hessian estimates for the regularized equation, and explicit representation formulas such as the Hopf-Lax formula. We refer, for example, to \cite{Fujita05, Stromberg10} for Hessian estimates for viscous Hamilton-Jacobi equations. A game-theoretic approach to semiconvexity for fully nonlinear evolution equations was developed in \cite{LiuSemiconvexity}, and recent one-sided second order estimates based on nonlinear adjoint methods can be found in \cite{CGM}.

Much less seems to be known in the presence of boundary conditions. For state-constraint problems, Han \cite{HanSemiconcavity} recently established global semiconcavity for a class of first order Hamilton-Jacobi equations using optimal-control and weak KAM methods. For Neumann problems, Mitake and Ni studied quantitative homogenization in \cite{MNhom} and semiconcavity estimates in \cite{MNsemi}. Goffi \cite{GoffiNeumann} obtained quantitative vanishing-viscosity estimates for time-dependent Hamilton-Jacobi equations with homogeneous Neumann boundary conditions on convex domains, including one-sided second order estimates. Related geometric estimates for the heat equation with homogeneous Neumann boundary conditions were obtained in \cite{GoffiTralli}. These works are complementary to the present paper. Our focus here is the preservation of convexity, semiconvexity, and strong convexity for first order Hamilton-Jacobi equations with Neumann-type boundary conditions.

The paper is organized as follows. In Section~\ref{sec:prelim}, we recall the viscosity formulations of the Neumann and state-constraint problems. In Section~\ref{sec:neumann-state}, we prove the Neumann-to-state-constraint property under the outward-normal monotonicity assumption on the Hamiltonian. In Section~\ref{sec:convexity}, we establish a generalized convex-envelope argument and derive the preservation of convexity, semiconvexity, and strong convexity. We also discuss nonhomogeneous Neumann conditions and a logarithmic transformation leading to Robin-type boundary conditions. Finally, in Section~\ref{sec:discussion}, we present several further observations and possible extensions. These include a data-dependent convexity-preservation result for first order Neumann problems under a relaxed version of the monotonicity assumption (H5) in Section~\ref{sec:data-dependent}, and a discussion of second order parabolic equations in one space dimension in  Section~\ref{sec:parabolic}. 

\subsection*{Acknowledgments}
The authors are grateful to Paolo Salani for helpful discussions.

\section{Preliminaries}\label{sec:prelim}

In this section, we briefly review some basics about viscosity solutions of \eqref{general eq1} with a general boundary condition
\begin{equation}\label{gen bdry}
B(x, t, u, \nabla u)=0 \quad \text{on $\partial\Omega \times (0, \infty)$,}
\end{equation}
where $B: \partial \Omega\times (0, \infty)\times \R\times \R^n\to \R$ is continuous and 
\begin{equation}\label{bdry-mono}
 \alpha \mapsto B(x, t, r, p+\alpha \nu(x)) \quad\text{is strictly increasing in $\R$.}
\end{equation}
Here, $\Omega\subset \R^n$ is a bounded domain of $C^2$ class. We assume that $H: \Oba\times [0, \infty)\times \R\times \R^n\to \R$ is continuous and satisfies (H1)--(H3). We also denote by $USC(\Oba)$ and $LSC(\Oba)$ the sets of upper semicontinuous and lower semicontinous functions on $\Oba$, respectively. 

\begin{defi}\label{def:neumann}
(1) A locally bounded function $u\in USC(\Oba\times (0, \infty))$ is called a viscosity subsolution of \eqref{general eq1}\eqref{gen bdry} if whenever there exist $(x_0, t_0)\in \Oba\times (0, \infty)$ and $\varphi\in C^1(\Oba\times (0, \infty))$ such that $u-\varphi$ attains a local maximum in $\Oba\times (0, \infty)$ at $(x_0, t_0)$, the following conditions hold:
\begin{enumerate}
    \item[1)] If $x_0\in \Omega$, then 
    \begin{equation}\label{vis-sub1}
        \varphi_t(x_0, t_0)+H(x_0, t_0, u(x_0, t_0), \nabla \varphi(x_0, t_0))\leq 0.
    \end{equation}
    \item[2)] If $x_0\in \partial \Omega$, then either \eqref{vis-sub1} holds or
    \begin{equation}\label{vis-sub2}
       B(x_0, t_0, u(x_0, t_0), \nabla \varphi(x_0, t_0))\leq 0
    \end{equation}
\end{enumerate}
(2) A locally bounded function $u\in LSC(\Oba\times (0, \infty))$ is called a viscosity supersolution of \eqref{general eq1}\eqref{gen bdry} if whenever there exist $(x_0, t_0)\in \Oba\times (0, \infty)$ and $\varphi\in C^1(\Oba\times (0, \infty))$ such that $u-\varphi$ attains a local minimum in $\Oba\times (0, \infty)$ at $(x_0, t_0)$, the following conditions hold:
\begin{enumerate}
    \item[1)] If $x_0\in \Omega$, then 
    \begin{equation}\label{vis-super1}
        \varphi_t(x_0, t_0)+H(x_0, t_0, u(x_0, t_0), \nabla \varphi(x_0, t_0))\geq 0.
    \end{equation}
    \item[2)] If $x_0\in \partial \Omega$, then either \eqref{vis-super1} holds or
    \begin{equation}\label{vis-super2}
    B(x_0, t_0, u(x_0, t_0), \nabla \varphi(x_0, t_0))\geq 0
    \end{equation}   
\end{enumerate}
(3) A function $u\in C(\Oba\times (0, \infty))$ is called a viscosity solution of \eqref{general eq1}\eqref{gen bdry} if it is both a viscosity subsolution and a viscosity supersolution of \eqref{general eq1}\eqref{gen bdry}. If in addition $u\in C(\Oba\times [0, \infty))$ and satisfies \eqref{initial} with $u_0\in C(\Oba)$, then $u$ is called a viscosity solution of \eqref{general eq1}, \eqref{initial} and \eqref{gen bdry}.
\end{defi}

As is standard in viscosity solution theory, in the above definition one may replace the local maximum (resp., local minimum) of $u-\varphi$ by a global maximum (resp., global minimum), or by a strict global maximum (resp., strict global minimum).

It is standard that the definition of viscosity solutions can equivalently
be formulated in terms of semijets \cite{CIL}. Recall that 
for $u\in USC(\overline\Omega\times(0,T))$ upper semicontinuous, the superjet $\mathcal P_{\overline\Omega}^{1,+}u(x,t)$ of $u$ at $(x,t)$ relative to $\overline\Omega$ is 
\[
\mathcal P_{\overline\Omega}^{1,+}u(x,t)
=
\left\{
(a,p)\in\mathbb R\times\mathbb R^n:
\begin{array}{l}
u(y,s)\leq u(x,t)+a(s-t)+\langle p,y-x\rangle\\[1mm]
\qquad\qquad +\,o(|s-t|+|y-x|)
\end{array}
\right\},
\]
where $(y,s)\to(x,t)$ with
$(y,s)\in\overline\Omega\times(0,T)$.  For 
$u\in LSC(\overline\Omega\times(0,T))$, the subjet $\mathcal P_{\overline\Omega}^{1,-}u(x,t)$ is
defined by reversing the inequality.  Equivalently,
\[
(a,p)\in\mathcal P_{\overline\Omega}^{1,-}u(x,t)
\quad \text{if and only if }\quad 
(-a,-p)\in\mathcal P_{\overline\Omega}^{1,+}(-u)(x,t).
\]
We denote by
$\overline{\mathcal P}_{\overline\Omega}^{1,+}u(x,t)$
the closure of the superjet in the following sense:
$(a,p)\in
\overline{\mathcal P}_{\overline\Omega}^{1,+}u(x,t)$
if there exist sequences $(x_j,t_j)\in\overline\Omega\times(0,T), (a_j,p_j)\in
\mathcal P_{\overline\Omega}^{1,+}u(x_j,t_j)$ 
such that
\[
(x_j,t_j,u(x_j,t_j),a_j,p_j)
\to 
(x,t,u(x,t),a,p)
\]
as $j\to\infty$.
Similarly,
$(a,p)\in
\overline{\mathcal P}_{\overline\Omega}^{1,-}u(x,t)$
if there exist $(x_j,t_j)\in\overline\Omega\times(0,T), 
(a_j,p_j)\in
\mathcal P_{\overline\Omega}^{1,-}u(x_j,t_j)$ such that
\[
(x_j,t_j,u(x_j,t_j),a_j,p_j)
\to
(x,t,u(x,t),a,p).
\]
By the continuity of the equation and the boundary operator, the
viscosity inequalities in Definition~\ref{def:neumann} may equivalently
be imposed on the corresponding closed semijets. In particular, if $u$
is a viscosity supersolution, the supersolution inequalities remain
valid for every $(a,p)\in
\overline{\mathcal P}_{\overline\Omega}^{1,-}u(x,t)$,
with the usual viscosity boundary alternative when 
$x\in\partial\Omega$. The analogous statement holds for
subsolutions and
$\overline{\mathcal P}_{\overline\Omega}^{1,+}u(x,t)$.
See \cite{CIL}.

We shall use the standard comparison principle for the Cauchy--Neumann problem. Under the standing assumptions of this paper, namely that $\Omega$ is a
bounded $C^2$ domain, $g$ is sufficiently regular, and $H$ satisfies
(H1)--(H3), the comparison principle for
\eqref{general eq1}--\eqref{general eq2} follows from the standard
viscosity theory for Hamilton-Jacobi equations with Neumann boundary
conditions; see, for example, \cite{Lio1,BL,CIL}.

We emphasize that (H1)--(H3) are only one convenient sufficient framework for
comparison.  Other comparison theories may be used instead, such as the framework developed by \cite{IshiiWKAM}, \cite{BIM} and related works. The convexity
mechanism in this paper applies in settings where comparison is
available under structural assumptions different from (H1)--(H3), provided
the hypotheses used in the Neumann-to-state-constraint lemma are satisfied.

Our analysis later also involves the so-called state constraint problem associated to \eqref{general eq1}. Below we quickly review the definition of viscosity solutions to the state constraint problem. 

\begin{defi}
    A function $u\in C(\Oba\times (0, \infty))$ is called a constrained viscosity solution if both
\begin{equation}\label{constraint-sub}
u_t+ H(x, t, u, \nabla u)\leq 0 \quad \text{in $\Omega\times (0, \infty)$}
\end{equation}
and 
\begin{equation}\label{constraint-super}
u_t+ H(x, t, u, \nabla u)\geq 0 \quad \text{in $\Oba\times (0, \infty)$}
\end{equation}
hold in the viscosity sense. More precisely, the interior subsolution property \eqref{constraint-sub} is defined as in a) of Definition~\ref{def:neumann}(1), and $u$ is said to satisfy \eqref{constraint-super} in the viscosity sense (or called a constrained viscosity supersolution) if  \eqref{vis-super1} holds whenever there exist $(x_0, t_0)\in \Oba\times (0, \infty)$ and $\varphi\in C^1(\Oba\times (0, \infty))$ such that $u-\varphi$ attains a local minimum in $\Oba\times (0, \infty)$ at $(x_0, t_0)$. 
\end{defi}

The state-constraint supersolution condition
\eqref{constraint-super} can also be equivalently expressed in terms of
relative subjets. Namely, $u$ is a constrained viscosity supersolution
if and only if
\[
a+H(x,t,u(x,t),p)\geq 0
\]
for every $(x,t)\in\overline\Omega\times(0,\infty)$ and every
$(a,p)\in
\mathcal P_{\overline\Omega}^{1,-}u(x,t)$. Moreover, by continuity of $H$, the same inequality holds for every $(a,p)\in
\overline{\mathcal P}_{\overline\Omega}^{1,-}u(x,t)$. Thus, $u$ is a constrained viscosity supersolution
if and only if 
\begin{equation}\label{sc-equiv}
a+H(x,t,u(x,t),p)\geq0
\quad
\text{for all }
(a,p)\in
\overline{\mathcal P}_{\overline\Omega}^{1,-}u(x,t).
\end{equation}

\section{Neumann boundary conditions and state constraints}\label{sec:neumann-state}

Our convexity preserving property is based on Lemma~\ref{lem:constraint}, which reveals that for a convex domain the Neumann boundary conditions actually implies state constraint condition if (H5) holds.
We prove Lemma \ref{lem:constraint} in this section. Let $d: \Oba\to \R$ be a $C^2$  function such that 
\begin{equation}\label{d-def}
d(x)= -\dist(x, \R^n\setminus \Omega), \quad \text{for all $x\in \N\cap \Oba$}, 
\end{equation}
where $\N\subset \R^n$ is a neighborhood of $\partial \Omega$. 
It is clear that $\nabla d(x)=\nu(x)$ for $x\in \partial \Omega$.

\begin{lem}\label{lem:constraint}
    Let $\Omega\subset \R^n$ be a bounded convex domain of $C^2$ class with outward unit normal $\nu$. Assume that $H\in C(\Oba\times [0, \infty)\times \R\times \R^n)$ satisfies  (H5). Let $u\in LSC(\Oba\times (0, \infty))$ be a viscosity supersolution of \eqref{general eq1}\eqref{general eq2}. Then it satisfies \eqref{sc-form} in the viscosity sense. 
\end{lem}
\begin{proof}
    Suppose that there exist $\varphi\in C^1(\Oba\times (0, \infty))$ and $(\hat{x}, \hat{t})\in \Oba\times (0, \infty)$ such that $u-\varphi$ attains a strict minimum over $\Oba\times (0, \infty)$ at $(\hat{x}, \hat{t})\in \partial \Omega\times (0, \infty)$. 
    
    We adopt the doubling variable technique, considering for each $\vep>0$ small
    \[
    \Phi_\vep(x, y, t)=u(x, t)-\varphi(y, t)+\frac{1}{\vep}|x-y|^2-G_\vep(x)+G_\vep(y)
    \]
    for $x, y\in \Oba$ and $t>0$, where $G_\vep\in C^2(\Oba)$ is taken such that
    \begin{enumerate}
        \item[(a)] $\sup_{\Oba}|\nabla G_\vep|\leq C$ and $\sup_{\Oba}|\nabla^2 G_\vep|\leq C\vep^{-1/2}$ for some $C>0$ independent of $\vep$;
        \item[(b)] $G_\vep\to 0$ uniformly in $\Oba$ as $\vep\to 0$;  
        \item[(c)] For any compact subset $K\subset \Omega$, $|\nabla G_\vep|\to 0$ uniformly in $K$ as $\vep\to 0$.
        \item[(d)] $\la \nabla G_\vep, \nu\ra<g$ on $\partial \Omega$ and $\nabla G_\vep\to g \nu$ uniformly on $\partial \Omega$ as $\vep\to 0$;      
    \end{enumerate}
    
We now construct such a family.  Fix a $C^2$ extension
$\widetilde g$ of $g$ to a tubular neighborhood of $\partial\Omega$.
Such an extension exists by the standard extension theorem for functions on a
compact $C^2$ hypersurface.  Choose $\chi\in C^\infty([0,\infty))$ such that
\[
0\leq\chi\leq1,\qquad
\chi(s)=1\ \text{for }0\leq s\leq1,\qquad
\chi(s)=0\ \text{for }s\geq2.
\]
For $\varepsilon>0$ sufficiently small, define
\begin{equation}\label{G-eps-def}
G_\varepsilon(x)
=
(\widetilde g(x)-\varepsilon)\,
d(x)\,
\chi\!\left(-\frac{d(x)}{\sqrt{\varepsilon}}\right)
\end{equation}
in the tubular neighborhood and set $G_\varepsilon=0$ where
$-d(x)\geq2\sqrt{\varepsilon}$.  This gives a $C^2$ function on
$\overline\Omega$.

At a boundary point $x\in\partial\Omega$, we have $d(x)=0$,
$\nabla d(x)=\nu(x)$ and $\chi(0)=1$, and hence
\begin{equation}\label{G-boundary-gradient}
\nabla G_\varepsilon(x)=(g(x)-\varepsilon)\nu(x).
\end{equation}
Thus
\[
\langle\nabla G_\varepsilon(x),\nu(x)\rangle
=g(x)-\varepsilon<g(x),
\]
and $\nabla G_\varepsilon\to g\nu$ uniformly on $\partial\Omega$.

Since $G_\varepsilon$ is supported in a boundary layer of width
$O(\sqrt{\varepsilon})$, direct differentiation of
\eqref{G-eps-def} gives
\[
\|G_\varepsilon\|_{L^\infty(\Omega)}\leq C\sqrt{\varepsilon},
\qquad
\|\nabla G_\varepsilon\|_{L^\infty(\Omega)}\leq C,
\qquad
\|\nabla^2G_\varepsilon\|_{L^\infty(\Omega)}
\leq C\varepsilon^{-1/2}.
\]
Moreover, for every compact $K\subset\Omega$, one has
$G_\varepsilon=0$ on $K$ for all sufficiently small $\varepsilon$.
Hence properties (a)--(d) above are satisfied.

We next take $J$ to be a bounded open interval in $(0, \infty)$ such that $\hat{t}\in J$. Let $(x_\vep, y_\vep, t_\vep)$ be a minimizer of $\Phi_\vep$ over $\Oba^2\times \ol{J}$. It is then easily seen that 
\[
\Phi_\vep(x_\vep, y_\vep, t_\vep)\leq \Phi_\vep(\hat{x}, \hat{x}, \hat{t})=u(\hat{x}, \hat{t})-\varphi(\hat{x}, \hat{t}), 
\]
which implies 
\begin{equation}\label{constraint eq1}
\begin{aligned}
\frac{1}{\vep}|x_\vep-y_\vep|^2\leq u(\hat{x}, \hat{t}) -\varphi(\hat{x}, \hat{t})-u(x_\vep, t_\vep)+\varphi(y_\vep, t_\vep)+G_\vep(x_\vep)-G_\vep(y_\vep).    
\end{aligned}
\end{equation}
By compactness of $\overline\Omega$ and the boundedness of the terms on the
right-hand side of \eqref{constraint eq1}, there is $M>0$ such that
\[
|x_\varepsilon-y_\varepsilon|\leq M\sqrt{\varepsilon}.
\]
After passing to a subsequence,
\[
x_\varepsilon\to x_0,\qquad
y_\varepsilon\to x_0,\qquad
t_\varepsilon\to t_0
\]
for some $(x_0,t_0)\in\overline\Omega\times\overline J$.
Using the lower semicontinuity of $u$, property (b) of $G_\varepsilon$, and
\eqref{constraint eq1}, we get
\[
0\leq
u(\hat x,\hat t)-\varphi(\hat x,\hat t)
-u(x_0,t_0)+\varphi(x_0,t_0).
\]
Since $u-\varphi$ has a strict minimum at $(\hat x,\hat t)$, this implies $(x_0,t_0)=(\hat x,\hat t)$. Returning once more to \eqref{constraint eq1}, lower semicontinuity yields $u(x_\varepsilon,t_\varepsilon)\to u(\hat x,\hat t)$, and therefore
\begin{equation}\label{constraint eq3}
\frac{|x_\varepsilon-y_\varepsilon|^2}{\varepsilon}\to0.
\end{equation}
In particular, we have $|x_\varepsilon-y_\varepsilon|=o(\sqrt{\varepsilon})$.

In view of the minimum of $\Phi_\vep$ at $(x_\vep, y_\vep, t_\vep)\in \Oba^2\times (0, \infty)$, we deduce the following two conditions:
\begin{enumerate}
    \item[i)] $(x, t)\mapsto u(x, t)-\psi(x, t)$ attains a local minimum at $(x_\vep, t_\vep)$ over $\Oba\times (0, \infty)$, where
    \[
    \psi(x, t)=\varphi(y_\vep, t)-\frac{1}{\vep}|x-y_\vep|^2+G_\vep(x)-G_\vep(y_\vep).
    \]
    \item[ii)] $y\mapsto \varphi(y, t_\vep)-\eta(y)$ attains a local maximum at $y_\vep\in \Oba$, where 
    \[
    \eta(y)=\frac{1}{\vep}|x_\vep-y|^2-G_\vep(x_\vep)+G_\vep(y). 
    \]
\end{enumerate}

By direct computations, we get
\begin{equation}\label{ishii-t}
    \partial_t \psi(x_\vep, t_\vep)= \partial_t \varphi(y_\vep, t_\vep),
\end{equation}
\begin{equation}\label{ishii-x}
    \nabla\psi(x_\vep, t_\vep)= \frac{2}{\vep}(y_\vep-x_\vep)+\nabla G_\vep(x_\vep), \quad \nabla \eta(y_\vep)=\frac{2}{\vep}(y_\vep-x_\vep)+\nabla G_\vep(y_\vep).
\end{equation}

If $x_\vep\in \Omega$, then by applying the definition of viscosity supersolutions of \eqref{general eq1} to $u$, we obtain 
\begin{equation}\label{constraint eq2}
\partial_t \psi(x_\vep, t_\vep)+ H(x_\vep, t_\vep, u(x_\vep, t_\vep), \nabla \psi(x_\vep, t_\vep))\geq 0. 
\end{equation}
In the case of $x_\vep\in \partial \Omega$, since $\la \nabla G_\vep, \nu\ra<g$ on $\partial \Omega$ as listed in property (d) of $G_\vep$, by \eqref{ishii-x} and the convexity of $\Omega$ we have 
\[
\la \nabla \psi(x_\vep, t_\vep), \nu(x_\vep)\ra= \frac{2}{\vep}\la y_\vep-x_\vep, \nu(x_\vep)\ra +\la \nabla G_\vep(x_\vep), \nu(x_\vep)\ra<g(x_\vep). 
\]
We thus deduce \eqref{constraint eq2} again by using the definition of viscosity supersolutions at boundary points. 

Since the state-constraint supersolution inequality is already part of the
usual supersolution condition at interior points, it remains only to treat
$\hat x\in\partial\Omega$.  We distinguish two cases according to the
position of $y_\varepsilon$.

Our goal is to prove
\begin{equation}\label{constraint eq4}
\partial_t\varphi(\hat x,\hat t)
+H(\hat x,\hat t,u(\hat x,\hat t),\nabla\varphi(\hat x,\hat t))
\geq0.
\end{equation}
If $y_\vep\in \Omega$ for all $\vep>0$ small, then, by \eqref{ishii-x} we have 
\[
\nabla \varphi(y_\vep, t_\vep)-\nabla \psi(x_\vep, t_\vep)=\nabla G_\vep(y_\vep)- \nabla G_\vep(x_\vep).
\]
Since $\sup_{\Oba}|\nabla^2 G_\vep|\leq C\vep^{-1/2}$ as stated in property (a) of $G_\vep$ and $|x_\vep-y_\vep|=o(\vep^{1/2})$, we are led to 
\[
|\nabla G_\vep(x_\vep)- \nabla G_\vep(y_\vep)|\to 0\quad \text{as $\vep\to 0$,}
\]
Adopting \eqref{ishii-t} and \eqref{constraint eq2}, and the continuity of $H$, we are thus led to 
  \eqref{constraint eq4} again by  sending $\vep\to 0$  . 

It remains to consider a subsequence for which
$y_\varepsilon\in\partial\Omega$.  Since
$y\mapsto\varphi(y,t_\varepsilon)-\eta(y)$ has a local maximum over
$\overline\Omega$ at $y_\varepsilon$, the tangential derivatives agree and
the one-sided normal derivative gives
\begin{equation}\label{normal-contact}
\nabla\varphi(y_\varepsilon,t_\varepsilon)
=
\nabla\eta(y_\varepsilon)+c_\varepsilon\nu(y_\varepsilon),
\quad c_\varepsilon\geq0.
\end{equation}
The convexity of $\Omega$ yields $\langle y_\varepsilon-x_\varepsilon,\nu(y_\varepsilon)\rangle\geq0$. Using \eqref{ishii-x} and \eqref{G-boundary-gradient}, we obtain
\[
\langle\nabla\eta(y_\varepsilon),\nu(y_\varepsilon)\rangle
=
\frac{2}{\varepsilon}
\langle y_\varepsilon-x_\varepsilon,\nu(y_\varepsilon)\rangle
+g(y_\varepsilon)-\varepsilon
\geq g(y_\varepsilon)-\varepsilon.
\]
It follows that
\begin{align}
    0\leq c_\varepsilon \leq ||\nabla \varphi||_{L^\infty(\overline{\Omega}\times J)}+||g||_{L^\infty(\partial\Omega)}+\varepsilon. 
\end{align}
Thus $c_\varepsilon, \nabla\psi(x_\varepsilon,t_\varepsilon),\nabla\eta(y_\varepsilon)$ are bounded. After extracting a further subsequence, we have $c_\varepsilon\rightarrow c\geq 0$ and $\nabla\psi(x_\varepsilon,t_\varepsilon),\nabla\eta(y_\varepsilon)\rightarrow q$, where $q\in \R^n$ satisfies 
\begin{align*}
    \nabla\varphi (\hat{x}, \hat{t})=q+c\nu (\hat{x}),\quad \langle q,\nu(\hat{x})\rangle\geq g(\hat{x}).
\end{align*}
Passing directly to the limit in \eqref{constraint eq2} gives
\begin{equation*}
\partial_t\varphi(\hat x,\hat t)
+H(\hat x,\hat t,u(\hat x,\hat t),q)
\geq0.
\end{equation*}
Finally, (H5) implies
\begin{align*}
    H(\hat x,\hat t,u(\hat x,\hat t),\nabla \varphi(\hat{x},\hat{t}))\geq H(\hat x,\hat t,u(\hat x,\hat t),q), 
\end{align*}
which yields \eqref{constraint eq4}. 
\end{proof}

When $g=0$, we can simply take $G_\vep$ to be
\begin{equation}\label{G-fun}
G_\vep(x)=-\vep d(x), \quad x\in \Oba. 
\end{equation}
One can easily verify that such a choice fulfills the conditions (a)--(d) imposed on $G_\vep$. 

\begin{rmk}
    The same proof as above (with $G_\vep$ given by \eqref{G-fun}) extends to the case of the fully nonlinear boundary condition \eqref{gen bdry} provided that $B\in C(\partial \Omega\times (0, \infty)\times \R\times\R^n)$ satisfies \eqref{bdry-mono} and 
    \begin{equation}
    B(x, t, r, p)< 0 \quad \text{whenever $\la p, \nu(x)\ra< 0$ holds.}
    \end{equation}
   This generalization allows for boundary conditions beyond \eqref{general eq2}, including the oblique boundary condition 
    \[
    \la \nabla u, \nu \ra-\beta|\nabla u|=0 \quad \text{on $\partial \Omega\times (0, \infty)$}
    \]
    with $\beta\in [0, 1)$.
\end{rmk}

\section{Convexity preserving property}\label{sec:convexity}

We now use Lemma~\ref{lem:constraint} to study preservation of generalized convexity for viscosity solutions to \eqref{general eq1}--\eqref{initial}. 

\subsection{ General Neumann boundary condition}
By Lemma~\ref{lem:constraint}, the Neumann problem can be treated as a
state-constraint problem at the supersolution level.  We may therefore adapt
the convex-envelope argument of \cite[Theorem~1]{ALL}.  The original result
there is stationary; the proposition below is the time-dependent first order
version needed here and also keeps track of a time-dependent convexity
parameter.

\begin{prop}[Supersolution property of generalized convex envelope]\label{prop:envelope}
Let $\Omega\subset\R^n$ be a bounded convex $C^2$ domain. Assume that the Hamiltonian 
is independent of $u$ and that
$H\in C(\Oba\times(0,\infty)\times\R^n)$ satisfies (H4$'$) and (H5).
Let $\mu\in C^1([0,T))$ satisfy \eqref{g-fun}. Let
$u\in LSC(\Oba\times(0,\infty))$ be a viscosity supersolution of
\eqref{constraint-super}. Fix
$\lambda=(\lambda_1,\ldots,\lambda_{n+1})$ satisfying $\lambda_i>0$ and
$\sum_{i=1}^{n+1}\lambda_i=1$. Define
\begin{equation}\label{envelope}
u_\ast(x,t)=\inf\left\{
\sum_{i=1}^{n+1}\lambda_i u(y_i,t)-\frac{\mu(t)}2Q(y_1,\ldots,y_{n+1}):
y_i\in\Oba,\ \sum_{i=1}^{n+1}\lambda_i y_i=x\right\},
\end{equation}
where
\[
Q(y_1,\ldots,y_{n+1})=
\sum_{i=1}^{n+1}\lambda_i|y_i|^2-
\left|\sum_{i=1}^{n+1}\lambda_i y_i\right|^2.
\]
Then $u_\ast$ satisfies \eqref{constraint-super} in the viscosity sense.
\end{prop}

\begin{proof}
Since $\overline\Omega$ is compact, the infimum in \eqref{envelope} is attained.
A standard compactness argument also shows that
$u_\ast\in LSC(\overline\Omega\times(0,T))$.

Let $\varphi\in C^1(\overline\Omega\times(0,T))$ and suppose that
$u_\ast-\varphi$ has a strict local minimum at $(\hat x,\hat t)$.
Choose $\hat y_i\in\overline\Omega$ such that
\begin{equation}\label{convex-proof eq1}
 \sum_{i=1}^{n+1}\lambda_i\hat y_i=\hat x,\qquad
 u_\ast(\hat x,\hat t)
 =\sum_{i=1}^{n+1}\lambda_i u(\hat y_i,\hat t)
 -\frac{\mu(\hat t)}2 Q(\hat y_1,\ldots,\hat y_{n+1}).
\end{equation}
Then
\[
 \Phi(y_1,\ldots,y_{n+1},t)
 :=\sum_{i=1}^{n+1}\lambda_i u(y_i,t)
 -\varphi\!\left(\sum_{i=1}^{n+1}\lambda_i y_i,t\right)
 -\frac{\mu(t)}2Q(y_1,\ldots,y_{n+1})
\]
has a local minimum at $(\hat y_1,\ldots,\hat y_{n+1},\hat t)$.

Applying the parabolic Crandall-Ishii lemma (the first order part of \cite[Theorem~8.3]{CIL}) to the functions $\lambda_i u$, we
obtain numbers $b_i$ such that
\[
 \sum_{i=1}^{n+1}b_i
 =\varphi_t(\hat x,\hat t)+\frac{\mu'(\hat t)}2Q(\hat y)
\]
and
\[
 \left(
 \frac{b_i}{\lambda_i},
 \nabla\varphi(\hat x,\hat t)+\mu(\hat t)(\hat y_i-\hat x)
 \right)
 \in\overline{\mathcal P}^{\,1,-}_{\overline\Omega}
 u(\hat y_i,\hat t).
\]
Here we used
\[
 \nabla_{y_i}Q(\hat y)=2\lambda_i(\hat y_i-\hat x).
\]
Since $u$ is a state-constraint supersolution by Lemma~\ref{lem:constraint}, the equivalent characterization \eqref{sc-equiv} with relative subjets yields
for every $i=1,\ldots,n+1$,
\[
 \frac{b_i}{\lambda_i}
 +H\!\left(\hat y_i,\hat t,
 \nabla\varphi(\hat x,\hat t)+\mu(\hat t)(\hat y_i-\hat x)\right)
 \ge0.
\]
Multiplying by $\lambda_i$ and summing yields
\[
0\le
 \varphi_t(\hat x,\hat t)
 +\frac{\mu'(\hat t)}2Q(\hat y)
+\sum_{i=1}^{n+1}\lambda_i
 H\left(\hat y_i,\hat t,
 \nabla\varphi(\hat x,\hat t)+\mu(\hat t)(\hat y_i-\hat x)\right).
\]
Applying (H4$'$) with $s=\mu(\hat t)$, we obtain
\[
0\le
 \varphi_t(\hat x,\hat t)
 +H(\hat x,\hat t,\nabla\varphi(\hat x,\hat t))
 +\left(\frac{\mu'(\hat t)}2+m(\mu(\hat t))\right)Q(\hat y).
\]
Since $Q\ge0$ and \eqref{g-fun} holds, it follows that
\[
 \varphi_t(\hat x,\hat t)
 +H(\hat x,\hat t,\nabla\varphi(\hat x,\hat t))
 \ge0.
\]
Thus $u_\ast$ is a viscosity supersolution of \eqref{constraint-super}.
\end{proof}

Let us now prove Theorem~\ref{thm:convex}. 
\begin{proof}[Proof of Theorem~\ref{thm:convex}] We extend the definition of $u_\ast(x, t)$ in \eqref{envelope} to all $x\in \Oba$ and all $t\geq 0$, rather than just $t>0$. By adopting the same argument in Step 1 of the proof of Proposition~\ref{prop:envelope} (with $(x^0, t^0)$ replaced by $(x^0, 0)$), we can show that $u_\ast\in LSC(\Oba\times [0, T))$. Moreover, since $u_0$ is assumed to be $\mu(0)$-convex in $\Oba$,  by definition we have $u_\ast(\cdot, 0)=u_0$ in $\Oba$. 

We have also shown in Lemma~\ref{lem:constraint} that $u$ satisfies \eqref{constraint-super} in the viscosity sense. By Proposition~\ref{prop:envelope}, we see that $u_\ast$ also solves \eqref{constraint-super} in the viscosity sense, which immediately implies $u_\ast$ is a viscosity supersolution of \eqref{general eq1}\eqref{general eq2}. Hence, by the comparison principle, we obtain $u_\ast\geq u$ in $\Oba\times [0, T)$. Clearly, $u_\ast\leq u$ in $\Oba\times [0, T)$. We therefore deduce that $u_\ast=u$ in $\Oba\times [0, T)$. Since $\lambda=(\lambda_1, \ldots, \lambda_{n+1})\in \R^{n+1}$ is arbitrary with positive components $\lambda_i$ satisfying $\sum_{i=1}^{n+1} \lambda_i=1$, this completes the proof of $\mu(t)$-convexity of $u(\cdot, t)$ in $\Oba$ for all $t>0$. 
\end{proof}

\begin{example}
    Let $a>0$. We focus on \eqref{general eq1}--\eqref{initial} in the case of $H(p)=(a^2+|p|^2)^{\frac{1}{2}}$ for $p\in \R^n$. In this special case, we have
    \begin{equation}
        \nabla H(p)=(a^2+|p|^2)^{-\frac{1}{2}}p,\quad \nabla^2 H(p)=-(a^2+|p|^2)^{-\frac{3}{2}}p\otimes p+(a^2+|p|^2)^{-\frac{1}{2}}I.  
    \end{equation}
In particular, 
\[
\nabla^2H(p)\leq (a^2+|p|^2)^{-\frac{1}{2}}I
\]
for all $p\in \R^n$, which verifies (H4') with $m(s)=s^2/(2a)$ for $s\in\R$. This enables us to find $\mu$ in \eqref{g-fun} by solving  
\begin{equation}\label{ex4-ode1}
    a \mu'+\mu^2=0\quad \text{in $(0, T)$}
\end{equation} 
with a given initial value $\mu(0)\in \R$. It is clear that $\mu\equiv 0$ is a solution in $(0, \infty)$ with $\mu(0)=0$. This corresponds to the standard convexity preserving property stated in Theorem~\ref{thm:convex0}. 

If $\mu(0)\neq 0$, a solution of \eqref{ex4-ode1} satisfies
\begin{equation}\label{ex4-ode1-sol}
\frac{1}{\mu(t)}=\frac{1}{\mu(0)}+\frac{t}{a}\quad \text{in $(0, T)$.}
\end{equation}
We have different behavior of $\mu$, depending on the sign of $\mu(0)$. In the case of $\mu(0)>0$, we can take $T=\infty$ and $\mu\in C^1([0, \infty))$ is positive and decreasing and satisfies $\mu(t)\to 0$ as $t\to \infty$. In other words, preservation of strong convexity holds, while the convexity parameter (lower bound of spatial Hessian) decays as $t\to \infty$. A concrete example that illustrates this behavior is the case of \eqref{ex2-eq} with $n=1$, $\Omega=(-1, 1)$ and
\[
u_0(x)=-(1-x^2)^{\frac{1}{2}} \quad x\in [-1, 1]. 
\]
The unique viscosity solution to this problem is 
\[
u(x, t)=-((1+t)^2-x^2)^{\frac{1}{2}} \quad x\in [-1, 1],\ t\geq 0.
\]
This solution can also be understood from the geometric interpretation of the equation: the graph of $u(\cdot, t)$ evolves downward with unit normal velocity. Clearly, in $(-1, 1)$, $u(\cdot, t)$ stays strongly convex for all $t>0$ and $u_{xx}(\cdot, t)\to 0$ as $t\to \infty$. 

In the case $\mu(0)<0$, we obtain from \eqref{ex4-ode1-sol} 
\[
T=-\frac{a}{\mu(0)}, \quad \mu(t)\to -\infty \quad \text{as $t\to T-.$}
\]
Although $\mu(t)$ only represents a lower bound of $u_{xx}(\cdot, t)$, it suggests a possible loss of semiconvexity along the evolution. This phenomenon is also known to occur for initial value problems (without boundary conditions). A typical example of such solutions of \eqref{ex2-eq} with $\Omega=\R$ is 
\[
u(x, t)=\begin{cases}
    ((t-1)^2-x^2)^{\frac{1}{2}}, & \text{for $|x|\leq (1-t)/\sqrt{2}$},\ t\leq 1,\\
    \sqrt{2}(1-t)-|x|, & \text{for $|x|>(1-t)/\sqrt{2}$.}
\end{cases}
\]
Note that in this case $u(\cdot, 0)$ is semiconvex but $u_{xx}(0, t)\to -\infty$ as $t\to 1-$. 

\end{example}

We remark that it is possible to extend Theorem~\ref{thm:convex} to Hamiltonians depending also on $u$ under a corresponding variant of (H4$'$).  More precisely,
assume that there exist $T>0$ and a continuous function $m:\R\to\R$ such that
\begin{equation}\label{H-concave-u}
\begin{aligned}
\sum_{i=1}^{n+1}\lambda_i
 H(x_i,t,r_i,p+s(x_i-x))
\leq {}& H\left(x,t,
 \sum_{i=1}^{n+1}\lambda_i r_i-\frac{s}{2}Q,p\right)+m(s)Q,
\end{aligned}
\end{equation}
for all $s\in\R$, $x_1,\ldots,x_{n+1}\in\Oba$,
$r_1,\ldots,r_{n+1}\in\R$, $t\in[0,T)$ and $p\in\R^n$, where
\[
x=\sum_{i=1}^{n+1}\lambda_i x_i,\quad
Q=\sum_{i=1}^{n+1}\lambda_i|x_i|^2-|x|^2,
\quad
\lambda_i\geq0,\quad \sum_{i=1}^{n+1}\lambda_i=1.
\]
Then Proposition~\ref{prop:envelope} and Theorem~\ref{thm:convex} remain valid,
with $H\in C(\Oba\times[0,\infty)\times\R\times\R^n)$ and with (H1)--(H3)
and (H5) understood in their original $r$-dependent form.

The only difference in the proof occurs when the Hamiltonian inequalities at
the minimizing points are combined.  If $r_i=u(\hat y_i,\hat t)$ and
$s=\mu(\hat t)$, then \eqref{convex-proof eq1} gives
\[
\sum_{i=1}^{n+1}\lambda_i r_i-\frac{s}{2}Q(\hat y)
=u_\ast(\hat x,\hat t).
\]
Consequently, applying \eqref{H-concave-u} to the sum of the viscosity
inequalities yields the Hamiltonian
$H(\hat x,\hat t,u_\ast(\hat x,\hat t),\nabla\varphi(\hat x,\hat t))$,
which is exactly the term required for the supersolution inequality of the
generalized convex envelope.  The rest of the proof is unchanged.  When $H$
is independent of $u$, \eqref{H-concave-u} reduces to (H4$'$).

For the propagation of semiconvexity, condition
\eqref{H-concave-u} can be slightly relaxed by using {\rm (H1)}.
Indeed, suppose that, for $s\leq 0$,
\begin{equation}\label{H-concave-u-semi}
\sum_{i=1}^{n+1}\lambda_i
H(x_i,t,r_i,p+s(x_i-x))
\leq
H\left(x,t,\sum_{i=1}^{n+1}\lambda_i r_i,p\right)
+m(s)Q,
\end{equation}
where
\[
x=\sum_{i=1}^{n+1}\lambda_i x_i,
\qquad
Q=\sum_{i=1}^{n+1}\lambda_i|x_i|^2-|x|^2.
\]
By {\rm (H1)}, since $s\leq0$ and $Q\geq0$,
\[
\begin{aligned}
H\left(x,t,\sum_i\lambda_i r_i,p\right)
&\leq
H\left(
x,t,\sum_i\lambda_i r_i-\frac{s}{2}Q,p
\right)
+\frac{\gamma s}{2}Q.
\end{aligned}
\]
Consequently, \eqref{H-concave-u} holds for $s\leq0$ with
$m(s)$ replaced by
\[
\widetilde m(s)=m(s)+\frac{\gamma s}{2}.
\]
It follows that, as long as $\mu(t)\leq0$, the same generalized
convex-envelope argument yields propagation of semiconvexity provided
\[
\frac12\mu'(t)+m(\mu(t))
+\frac{\gamma}{2}\mu(t)\leq0.
\]
Thus, in the semiconvex regime, the dependence of $H$ on $u$ can be
incorporated directly through the properness condition {\rm (H1)}.

Regarding ordinary convexity result (with $\mu=0$) in Theorem~\ref{thm:convex0}, the stronger condition (H4$'$) is not needed. Taking $\mu\equiv0$ in the envelope argument and assuming only (H4), we obtain the following simpler version of Proposition~\ref{prop:envelope}.

\begin{prop}[Supersolution property of convex envelope]\label{prop:envelope0}
    Let $\Omega\subset \R^n$ be a bounded convex $C^2$ domain. Assume that $H\in C(\Oba\times (0, \infty)\times \R\times \R^n)$ satisfies (H4)(H5). For $T>0$, let $u\in LSC(\Oba\times (0, T))$ be a viscosity supersolution of \eqref{constraint-super}. Fix $\lambda=(\lambda_1, \lambda_2, \ldots, \lambda_{n+1})\in \R^{n+1}$ satisfying $\lambda_i> 0$ for all $i=1, \ldots n+1$ and $\sum_{i=1}^{n+1}\lambda_i=1$. Let $u_\ast$ be the spatial convex envelope of $u$ in $\Oba\times (0, \infty)$ associated to $\lambda$, defined by
\begin{equation}\label{envelope0}
u_\ast(x, t)=\inf\left\{\sum_{i=1}^{n+1}\lambda_i u(y_i, t): y_i\in \Oba, \ \sum_{i=1}^{n+1} \lambda_i y_i=x\right\}
\end{equation}
for all $x\in \Oba, t\in (0, T)$. Then, $u_\ast$ is a viscosity supersolution of \eqref{constraint-super}.
\end{prop}

The same argument with $\mu\equiv0$ and (H4) gives
Theorem~\ref{thm:convex0}. Notice that Theorem~\ref{thm:convex0} allows dependence of $H$ on $u$, since preservation of the $0$-convexity requires a simpler concavity structure for $H$.

\subsection{Constant nonhomogeneous Neumann problem}\label{sec:general neumann}

In this section, we show how changes in the Neumann condition will affect the convexity preserving property. 

We examine the special case when $g\equiv k$ for a given $k\in \R$. In other words, we consider the Neumann boundary condition of the form
\begin{equation}\label{gen neumann}
    \la \nabla u, \nu\ra=k \quad \text{on $\partial \Omega\times (0, \infty)$},
\end{equation}
 In this case, (H5) reduces to the following condition.
\begin{enumerate}
    \item[(H5$_k$)] For every $(x,t,r)\in\partial\Omega\times(0,\infty)\times\R$,
every $p\in\R^n$ with $\langle p,\nu(x)\rangle\geq k$, and every
$\alpha\geq0$,
\[
H(x,t,r,p+\alpha\nu(x))\geq H(x,t,r,p).
\]
\end{enumerate}
It is easily seen that $(H5_k)$ is weaker when $k$ is larger. Hence, our results above show that, for general Hamilton-Jacobi equations of the form
\[
u_t+ H(x, t, \nabla u)=0 \quad \text{in $\Omega\times (0, \infty)$}
\]
with the Neumann boundary condition \eqref{gen neumann}, spatial convexity of the solution 
is more easily preserved when $k\geq 0$ is larger, in line with intuition. The following example illustrates the role of $k$ in convexity preservation explicitly. 

\begin{example}\label{exa3}
    Let $\Omega=B_1(0)\subset \R^n$. Consider 
    \begin{equation}\label{ex3-eq1}
    u_t +||\nabla u|-1|=0 \quad \text{in $B_1(0)$}
    \end{equation}
    with initial data $u_0$ as in \eqref{ex1-sol} and the boundary condition 
    \begin{equation}\label{ex3-eq2}
    \la \nabla u, \nu\ra=k \quad \text{on $\partial B_1(0)$.}
    \end{equation}
    For any $k\geq 1$, the unique viscosity solution is given by $u(x, t) =u_0(x)$ for $(x,t)\in \ol{B_1(0)}\times[0,\infty)$, and clearly convexity is preserved.  

    In contrast, when $k\in [0, 1)$, the unique viscosity solution is 
    \begin{equation}\label{ex3-eq3}
    u(x, t)= \min\{|x|, k|x|-(1-k)(t-1)\}, \quad x\in \ol{B_1(0)},\ t\geq 0.    
    \end{equation}
    Although the initial value is convex, the convexity (and even semiconvexity) of $u(\cdot, t)$ is lost for $0<t<1$. Such difference reflects the dependence of (H5$_k$) on $k$. Note that the Hamiltonian for \eqref{ex3-eq1} satisfies (H5$_k$) for $k\geq 1$, but not for $k\in [0, 1)$. 
\end{example}

\subsection{Log-concavity preservation with Robin boundary conditions}
Convexity preserving property for \eqref{general eq1}\eqref{gen neumann} can be interpreted as log-concavity preservation for a class of transformed first order equations with a Robin-type boundary condition. 
Note that, under the transformation 
\begin{equation}\label{transform}
w=e^{-u} \quad \text{in $\Oba\times [0, \infty)$,}
\end{equation}
 $u$ is a viscosity solution (resp., subsolution, supersolution) \eqref{general eq1}\eqref{gen neumann} in the viscosity sense if and only if 
 $w$ is a positive viscosity solution (resp., supersolution, subsolution) of 
\begin{equation}\label{robin-eq}
w_t+ \hat{H}(x, t, w, \nabla w)=0 \quad \text{in $\Omega\times (0, \infty)$}. 
\end{equation}
and 
\begin{equation}\label{robin-cond}
    k w+\la \nabla w, \nu\ra= 0\quad \text{on $\partial \Omega\times (0, \infty)$,}
\end{equation}
where 
\begin{equation}\label{transform-H}
\hat{H}(x, t, r, p):=-r H\left(x, t, -\log r, -\frac{p}{r}\right), \quad (x, t, r, p)\in \Oba\times (0, \infty)\times (0, \infty)\times \R^n.
\end{equation}
As a result, if convexity preservation holds for \eqref{general eq1}\eqref{gen neumann}, then log-concavity of initial value is preserved for positive solutions to \eqref{robin-eq}\eqref{robin-cond}. Hence,
\begin{cor}
    Let $\Omega\subset \R^n$ be a bounded convex domain of $C^2$ class with outward unit normal $\nu$ to $\partial \Omega$. Let $k\geq 0$. Suppose that $H\in C(\Oba\times [0, \infty)\times \R\times  \R^n)$ satisfies (H1)--(H4) and (H5$_k$). Let $\hat{H}$ be defined by \eqref{transform-H}. 
    Let $w_0\in C(\Oba)$ and let $w\in C(\Oba\times [0, \infty))$ be the unique positive viscosity solution of \eqref{robin-eq}\eqref{robin-cond} with initial value \eqref{initial}. If $w_0$ is log-concave in $\Oba$, then $w(\cdot, t)$ is also log-concave in $\Oba$ for all $t\geq 0$. 
\end{cor}

Under the assumptions (H1)--(H4) and (H5$_k$) on $H$, the new Hamiltonian $\hat{H}$ covers many examples. 

\begin{example}\label{exa4}
    Suppose that $H$ is given by 
\[
H(x, p)= a(x) |p|
\]
for $a\in C^{0,1}(\Oba)$, $a\geq 0$ in $\Oba$, we have 
\[
\hat{H}(x, p)= - a(x) |p|, \quad x\in \Oba,\ p\in \R^n.
\]
If in addition $a$ is assumed to be concave in $\Oba$, then by Theorem~\ref{thm:convex0} the convexity preserving property holds for 
\[
u_t+a(x) |\nabla u|=0 \quad \text{in $\Omega\times (0, \infty)$}
\]
with \eqref{gen neumann} for any $k\geq 0$. This in turn implies log-concavity preservation of positive solutions to 
\[
w_t-a(x) |\nabla w|=0 \quad \text{in $\Omega\times (0, \infty)$}
\]
associated with the Robin boundary condition \eqref{robin-cond} for all $k\geq 0$. 
\end{example}

\begin{example}\label{exa5}
    We can reinterpret Example~\ref{exa3} in terms of log-concavity preservation for a similar first order equation with a Robin boundary condition. Let $\Omega=B_1(0)\subset \R^n$ and $u\in C(\Oba\times [0, \infty))$ be the unique viscosity solution of \eqref{ex3-eq1}\eqref{ex3-eq2} with initial value $u_0$ in \eqref{initial} given by \eqref{ex1-sol}. By taking $w$ as in \eqref{transform}, we see that $w$ is the unique solution of 
    \[
    w_t-||\nabla w|-w|=0\quad \text{in $B_1(0)\times (0, \infty)$}
    \]
    satisfying Robin boundary condition \eqref{robin-cond} and initial value  
    \[
    w(x, 0)=e^{-|x|}, \quad x\in \ol{B_1(0)}. 
    \]
    In this new setting, Example~\ref{exa3} shows that log-concavity of $w$ is preserved when $k=1$ but is not preserved when $k=0$. In general, for $k\in [0, 1)$, the unique solution of the Cauchy-Robin problem is expressed as $w=e^{-u}$ with $u$ given by \eqref{ex3-eq3}. It clearly does not preserve log-concavity. 

    Interestingly, in this example we obtain the critical value $k=1$ for the Robin parameter, separating the regimes for preservation and loss of log-concavity along the evolution. This can be compared with results in \cite{ACH,CrFr6} for second order elliptic Robin problems, where log-concavity is established only when the Robin parameter is sufficiently large, with the threshold depending on the geometry and regularity of the domain. In contrast, in our first order example the critical value is determined only by the Hamiltonian.
\end{example}

\section{Further discussions}\label{sec:discussion}

\subsection{Convexity preservation depending on initial data}\label{sec:data-dependent}

Let us first make a remark regarding possible relaxation of (H5) in Theorem~\ref{thm:convex0}. 
The following example shows that the state-constraint reduction and (H5) are not necessary for convexity preservation with some initial values. 

\begin{example}
\label{exa:no-state-constraint}
Let $\Omega=(0,1)$ and consider the transport equation
\[
u_t+u_x=0
\quad \text{in }(0,1)\times(0,\infty),
\]
with homogeneous Neumann boundary conditions in the viscosity sense.
For the convex initial datum $u_0(x)=x^2$, $x\in (0, 1)$, 
the unique viscosity solution is
\[
u(x,t)=(x-t)_+^2,
\]
which is convex in $x$ for every $t\geq0$. However, $u$ is not a state-constraint supersolution. Indeed, for any
$t_0>0$, the function $\varphi(x,t)=-x$ touches $u$ from below at
$(0,t_0)$, while
\[
\varphi_t+H(\varphi_x)=-1<0,
\qquad H(p)=p.
\]
Therefore the state-constraint supersolution condition fails at $x=0$.
Notice also that (H5) fails there, since $\nu(0)=-1$ and, for $\alpha>0$, 
\[
H(p+\alpha\nu(0))=p-\alpha<p. 
\]
On the other hand, for the convex initial datum $u_0(x)=(1-x)^2$,
the solution is
\[
u(x,t)=\bigl(1-(x-t)_+\bigr)^2,
\]
which is not convex in $x$ for $0<t<1$. Thus, when (H5) is violated, convexity preservation may depend on the initial data. 
\end{example}

The above example suggests that a more general theory of convexity preservation involving initial data may be possible. Below we give a simple one-dimensional result showing that (H5) can be weakened when additional information on the initial data is available. We restrict ourselves to homogeneous Neumann boundary conditions and assume that the evolution preserves spatial monotonicity.

Let $\Omega=(0,1)$. We replace {\rm (H5)} by the following one-sided
condition:
\begin{equation}\label{one-side mono}
H(1,t,r,p+\alpha)\geq H(1,t,r,p)\quad \text{for all $(t,r)\in(0,\infty)\times\R$, $p\geq0$ and $\alpha\geq0$.}
\end{equation}
Notice that no corresponding monotonicity condition is imposed at $x=0$.

\begin{prop}\label{prop:data-dependent}
Let $\Omega=(0,1)$ and let $g=0$. Suppose that
$H\in C([0,1]\times[0,\infty)\times\R\times\R)$ satisfies
{\rm (H1)--(H4)} and \eqref{one-side mono}. Let $u\in C([0, 1]\times [0, \infty))$ be the unique 
viscosity solution of \eqref{general eq1}--\eqref{initial}.
Assume that $x\mapsto u(x, t)$ is 
nondecreasing in $[0,1]$ for all $[0, T)$ with $T>0$. If $u_0\in C([0,1])$ is convex and nondecreasing, then $u(\cdot, t)$ is convex in $[0,1]$ for every $t\in [0, T)$.
\end{prop}

\begin{proof}
Let $u_\ast(\cdot,t)$ denote the spatial convex envelope of
$u(\cdot,t)$ for all $t\in [0, T]$.  We claim that $u_\ast$ is a viscosity supersolution of the original
Neumann problem restricted to the time interval $(0, T)$. Let $\varphi$ touch $u_\ast$ from below at
$(\hat x,\hat t)$ with $\hat x\in(0,1)$ and $\hat{t}\in (0, T)$, and let $p=\varphi_x(\hat x,\hat t)$. Since $u_\ast(\cdot,\hat t)$ is convex and nondecreasing, we have $p\geq0$. Suppose that
\[
\hat x=\sum_i\lambda_i\hat y_i,
\qquad
u_\ast(\hat x,\hat t)
=\sum_i\lambda_i u(\hat y_i,\hat t),
\]
with $\lambda_i>0$ and $\sum_i\lambda_i=1$.

As in the proof of Proposition~\ref{prop:envelope0}, the Crandall-Ishii lemma
gives, at every point $\hat y_i$ realizing the convex envelope, a relative
subjet of $u$ whose spatial component is the same number $p$.
If $\hat y_i\in(0,1)$, the interior supersolution inequality applies.
If $\hat y_i=0$ and $p>0$, then the boundary term in the Neumann
supersolution condition is $-p<0$, and therefore the interior equation must again hold.

We argue separately for the case $p=0$. By spatial monotonicity of $u$, we get $u_\ast(x,\hat t)\geq u(0,\hat t)$ for all $x\in[0,1]$ 
and equality holds at $x=0$. Since $\varphi_x(\hat x,\hat t)=0$ and
$\varphi(\cdot,\hat t)$ touches the convex function
$u_\ast(\cdot,\hat t)$ from below at $\hat x$, $\hat x$ must be a minimizer of $u_\ast(\cdot,\hat t)$ and thus $u_\ast(\hat x,\hat t)=u(0,\hat t)$.

For the points $\hat y_i$ realizing the convex envelope, using the monotonicity of 
$u(\cdot,\hat t)$ again, we have $u(\hat y_i,\hat t)\geq u(0,\hat t)$
for every $i$. The relation $u_\ast(\hat x,\hat t)=u(0,\hat t)$ therefore implies $u(\hat y_i,\hat t)=u(0,\hat t)$ for every $i$. Since $\hat x\in (0, 1)$, at least one $\hat y_i\geq\hat x$,
and hence
\[
u(0,\hat t)\leq u(\hat x,\hat t)
\leq u(\hat y_i,\hat t)=u(0,\hat t).
\]
It follows that $u(\hat x,\hat t)=u_\ast(\hat x,\hat t)$. Since $u_\ast\leq u$, the function $\varphi$ also touches $u$ from
below at $(\hat x,\hat t)$. As $\hat x\in(0,1)$, the interior
supersolution condition gives
\[
\varphi_t(\hat x,\hat t)
+H(\hat x,\hat t,u_\ast(\hat x,\hat t),0)\geq0.
\]

Let us consider the case when $\hat y_i=1$. In this case, condition \eqref{one-side mono} allows us to repeat the
argument of Lemma~\ref{lem:constraint} only at the right endpoint and
only for nonnegative slopes. We therefore obtain the interior
supersolution inequality at $\hat y_i$ also in this case. Consequently, the same summation argument as in the proof of
Proposition~\ref{prop:envelope0}, together with {\rm (H4)}, shows that
\[
\varphi_t(\hat x,\hat t)
+H(\hat x,\hat t,u_\ast(\hat x,\hat t),p)\geq0.
\]
Thus $u_\ast$ satisfies the supersolution inequality in the interior.

It remains to consider the boundary points. Since
$u_\ast(0,t)=u(0,t)$ and $u_\ast(1,t)=u(1,t)$, the viscosity Neumann
supersolution condition follows directly from that of $u$, together with
the monotonicity and convexity of $u_\ast(\cdot,t)$. Hence $u_\ast$ is a viscosity supersolution of the original Neumann problem.

Since $u_\ast(\cdot,0)=u_0.$ thanks to the convexity of $u_0$, 
the comparison principle gives $u_\ast\geq u$ in $[0, 1]\times [0, T)$. By definition $u_\ast\leq u$. Therefore, $u_\ast=u$, and $u(\cdot,t)$ is convex for every $t\in [0, T)$.
\end{proof}

Condition \eqref{one-side mono} is weaker than {\rm (H5)}. However, in the present one-dimensional setting, convexity preservation is obtained only on the time interval $[0,T)$ during which $u$ remains nondecreasing in the spatial variable. It would be interesting to develop a more systematic theory of convexity preservation for the Neumann problem for general Hamilton-Jacobi equations in higher dimensions, identifying conditions on the initial value.

\subsection{Convexity preservation for parabolic equations}\label{sec:parabolic}

We conclude with a simple one-dimensional observation for second order
parabolic equations.  Its purpose is not to give a general parabolic theory,
but to indicate a contrast with the first order mechanism developed above.
For a genuinely parabolic equation, the boundary condition can impose an
additional compatibility requirement on convex initial data.  The following
example shows that, for a constant nonhomogeneous Neumann condition, a natural
one-sided compatibility condition is sufficient to propagate convexity.

The following one dimensional case provides a simple example illustrating such behavior. We assume that the domain is $(0, 1)\subset \R$. Let $k>0$ and $F:\R\to\R$ be a nonincreasing continuous function such that $F(s)=0$ if and only if $s=0$. Consider
\begin{numcases}{}
    u_t+F(u_{xx})=0 &\quad \text{in $(0, 1)\times (0, \infty)$,}\label{1d-para eq1}\\
     u_x(0, t)=-k, \quad u_x(1, t)=k &\quad \text{for $t\in (0, \infty)$,}\label{1d-para eq2}\\
    u(\cdot, 0)=u_0 &\quad \text{in $[0, 1]$,}\label{1d-para initial}
\end{numcases}
where $u_0\in C^1([0,1])$ is a given initial value. In this setting, the convexity-preserving property can be readily observed when $u_0$ satisfies 
\begin{equation}\label{compatibility}
    (u_0)_{x}(0)\geq -k,  \quad (u_0)_x(1)\leq k. 
\end{equation}
\begin{prop}
Let $k>0$. Let $F\in C(\R)$ be a nonincreasing function such that $F(s)=0$ if and only if $s=0$. Suppose that  $u_0\in C^1([0,1])$ is a convex function such that \eqref{compatibility} holds. 
Let $u$ be a unique viscosity solution of \eqref{1d-para eq1}--\eqref{1d-para initial}. Then, $u(\cdot, t)$ is convex in $[0, 1]$ for all $t\geq 0$.
\end{prop}
\begin{proof}
It is straightforward to verify that $u_0$ is a stationary subsolution of \eqref{1d-para eq1}--\eqref{1d-para initial}. By the comparison principle, we obtain $u_0(x)\leq u(x, t)$ for all $x\in [0, 1]\times [0, \infty)$. Note that for any $s\geq 0$, $(x, t)\mapsto u(x, t+s)$ is also a viscosity solution of \eqref{1d-para eq1}\eqref{1d-para eq2} with initial value $u(\cdot, s)$. Using the comparison principle again, we deduce $u(x, t)\leq u(x, t+s)$ for all $x\in (0, 1)$ and $t, s\geq 0$. 

Consequently, $u$ satisfies $u_t\geq 0$ in $[0, 1]\times (0, \infty)$ in the viscosity sense, which implies that $F(u_{xx}(\cdot, t))\leq 0$ holds in $(0, 1)$  for each $t>0$ also in the viscosity sense. Since $F$ is nonincreasing and $0$ is the only zero of $F$, we see that $-u_{xx}(\cdot, t)\leq 0$ in $(0, 1)$ holds in the viscosity sense for all $t>0$. By \cite[Lemma 1]{ALL}, this concludes the convexity of $u(\cdot, t)$ in $[0, 1]$ for all $t>0$. 
\end{proof}

\end{document}